\documentclass[11pt,a4paper]{amsart}

\usepackage{amsmath,amssymb,amsthm,amsfonts,mathtools}
\usepackage{geometry}
\usepackage{setspace}
\usepackage{enumitem}
\usepackage{xcolor}

\usepackage[colorlinks=true, citecolor=violet, linkcolor=blue, urlcolor=purple, bookmarks=false]{hyperref}
\usepackage{cleveref}
\crefformat{enumi}{(#2#1#3)}

\newcommand{\citeSta}[1]{\cite[\href{https://stacks.math.columbia.edu/tag/#1}{#1}]{stacks-project}}

\theoremstyle{plain}
\newtheorem{theorem}{Theorem}[section]
\newtheorem{lemma}[theorem]{Lemma}
\newtheorem{proposition}[theorem]{Proposition}
\newtheorem{corollary}[theorem]{Corollary}

\theoremstyle{definition}
\newtheorem{definition}[theorem]{Definition}
\newtheorem{remark}[theorem]{Remark}
\newtheorem{example}[theorem]{Example}
\newtheorem{notation}[theorem]{Notation}

\crefname{theorem}{Theorem}{Theorems}
\crefname{lemma}{Lemma}{Lemmas}
\crefname{proposition}{Proposition}{Propositions}
\crefname{corollary}{Corollary}{Corollaries}
\crefname{definition}{Definition}{Definitions}
\crefname{remark}{Remark}{Remarks}
\crefname{example}{Example}{Examples}
\crefname{notation}{Notation}{Notations}

\AddToHook{env/lemma/begin}{\crefalias{theorem}{lemma}}
\AddToHook{env/proposition/begin}{\crefalias{theorem}{proposition}}
\AddToHook{env/corollary/begin}{\crefalias{theorem}{corollary}}
\AddToHook{env/fact/begin}{\crefalias{theorem}{fact}}
\AddToHook{env/key/begin}{\crefalias{theorem}{key}}
\AddToHook{env/setting/begin}{\crefalias{theorem}{setting}}
\AddToHook{env/claim/begin}{\crefalias{theorem}{claim}}
\AddToHook{env/addendum/begin}{\crefalias{theorem}{addendum}}
\AddToHook{env/conjecture/begin}{\crefalias{theorem}{conjecture}}
\AddToHook{env/definition/begin}{\crefalias{theorem}{definition}}
\AddToHook{env/remark/begin}{\crefalias{theorem}{remark}}
\AddToHook{env/construction/begin}{\crefalias{theorem}{construction}}
\AddToHook{env/observation/begin}{\crefalias{theorem}{observation}}
\AddToHook{env/warning/begin}{\crefalias{theorem}{warning}}
\AddToHook{env/Notation/begin}{\crefalias{theorem}{Notation}}
\AddToHook{env/discussion/begin}{\crefalias{theorem}{discussion}}
\AddToHook{env/example/begin}{\crefalias{theorem}{example}}
\AddToHook{env/slogan/begin}{\crefalias{theorem}{slogan}}
\AddToHook{env/notation/begin}{\crefalias{theorem}{notation}}

\setlist[itemize]{labelsep=5pt}
\setlist[enumerate]{labelsep=5pt}
\setlist[description]{labelsep=11pt}

\newlist{enumalphp}{enumerate}{5}
\setlist[enumalphp]{label=(\alph*),labelsep=5pt}

\newcommand{\defeq}{\coloneqq}

\DeclareMathOperator{\pd}{pd}
\DeclareMathOperator{\Gpd}{Gpd}

\DeclareMathOperator{\Ext}{Ext}
\DeclareMathOperator{\Tor}{Tor}
\DeclareMathOperator{\Hom}{Hom}
\DeclareMathOperator{\depth}{depth}

\DeclareMathOperator{\egrade}{E.grade}

\DeclareMathOperator{\fdim}{fdim}
\DeclareMathOperator{\idim}{idim}
\DeclareMathOperator{\perf}{perf}

\DeclareMathOperator*{\colim}{colim}
\DeclareMathOperator{\Kos}{Kos}

\newcommand{\mcalD}{\mathcal{D}}

\newcommand{\mfrakm}{\mathfrak{m}}

\newcommand{\setZ}{\mathbb{Z}}

\newcommand{\setF}{\mathbb{F}}

\newcommand{\Rperf}{R_{\mathrm{perf}}}

\newcommand{\set}[2]{\{#1\mid#2\mbox{}\}}

\numberwithin{equation}{section}

\title[Homological Detection by Perfectoid Algebras]{Homological Detection by Perfectoid Algebras}

\author[M. Asgharzadeh]{Mohsen Asgharzadeh}
\address{Mohsen Asgharzadeh, Hakimiyeh, Tehran, Iran.}
\email{mohsenasgharzadeh@gmail.com}

\author[R. Ishizuka]{Ryo Ishizuka}
\address{Department of Mathematics, Institute of Science Tokyo, 2-12-1 Ookayama, Meguro, Tokyo 152-8551}
\email{ishizuka.r.ac@m.titech.ac.jp}



\begin{document}

	\begin{abstract}
		We first establish estimates for the projective and injective dimensions for quotients of perfectoid algebras by radical ideals. 
		As applications, we obtain homological characterizations of modules of finite injective dimension, (big) Cohen--Macaulay modules, Gorenstein local rings, and regular local rings in mixed characteristic in terms of perfectoid algebras. These are mixed characteristic analogues of the corresponding characterizations via Frobenius morphisms in positive characteristic.
	\end{abstract}

	\maketitle 
	
	\tableofcontents
	\section{Introduction}
	Let \(p\) be a prime number.
	Perfect and perfectoid algebras have become fundamental objects in modern commutative algebra.
	Besides their importance in $p$-adic geometry, they have proved remarkably effective in studying not only homological conjectures of Noetherian rings but also singularity theory in mixed characteristic.
	
	Since perfectoid rings in characteristic \(p\) are precisely perfect \(\setF_p\)-algebra, we can think perfectoid algebra as a mixed characteristic analogue of perfect rings, or more philosophically Frobenius morphisms.
	One of the most influential developments in this direction is the work of Bhatt, Iyengar and Ma \cite{bhatt2019Regular}, who showed that perfectoid algebras detect regularity through flat dimension.
	Their work can be viewed as a mixed-characteristic analogue of classical results of Kunz (\cite{kunz1969Characterizations}) and Peskine--Szpiro (\cite{peskine1973dimension}) involving the Frobenius endomorphism in positive characteristic.
	
	The success of perfectoid methods naturally raises the question of whether they can also detect
	other homological properties.
	Although flat dimension has been extensively studied in this context, much less is known about injective dimension, denoted by \(\idim_R(-)\),  and related dual homological invariants.
	The purpose of this paper is to develop such a theory.
	
	Our starting point is a finite projective and injective dimensions estimate for quotients of suitable perfectoid algebras.
	This presents a  generalization of \cite[Theorem 3.1]{aberbach1997Finite} and \cite[Lemma 3.2]{bhatt2019Regular}, in which they primarily treat flat dimension in the prime characteristic setting.
	See \Cref{RemarkProjDim} for a more detailed historical account of these results.
	
	\begin{corollary}[Special case of \Cref{TheoFinPdim} and \Cref{TheoFinIdim}]
		Let \(R\) be a Noetherian ring and let \(A\) be a perfectoid \(R\)-algebra.
		For any ideal \(I\) of \(R\)  such that \(p \in I\) and \(I\) is generated by \(r\) elements up to radical.
		Then
		\begin{enumerate}
			\item The projective dimension \(\pd_A (A/\sqrt{IA})\) is less than or equal to \(2(r+1)\).
			\item If \(\idim_R(A)\) is finite, then \(\idim_R(A/\sqrt{IA}) \leq \idim_R(A)\).
		\end{enumerate}
	\end{corollary}
	While perfectoid rings are typically not Noetherian, this result suggests that certain homological invariants of these highly non-Noetherian rings can be effectively controlled and estimated.
	
	In a foundational work, dual to that of  \cite{peskine1973dimension}, Herzog \cite[Satz 5.2]{herzog1974Ringe} established the following celebrated criterion: if \(R\) is an \(F\)-finite local ring of characteristic \(p\) and \(M\) is a finitely generated \(R\)-module such that \(\Ext_R^i({}^e\!R, M) = 0\) for all \(i > 0\) and infinitely many \(e\), then \(M\) has finite injective dimension.
	This equivalence was strengthened by Takahashi--Yoshino (\cite[Theorem 4.5]{takahashi2004Characterizing}) in which they only require the vanishing for \(i \gg 0\) for a sufficiently large \(e > 0\).
	
	As a mixed-characteristic analogue, part of our result gives the following characterization by using transcendental extension instead of such finite components. While we only state the following results in terms of perfectoid algebra for simplicity, the perfectoid condition can be weakened (\Cref{GeneralPerfectoidCondition}).
	
	\begin{theorem} (\Cref{TheoEquivGor}). \label{MainTheoremInjDim}
		Let \((R, \mfrakm)\) be a Noetherian local ring with \(p \in \mfrakm\), and let \(M\) be a nonzero finitely generated or \(\mfrakm\)-adically complete \(R\)-module. Then the following are equivalent:
		\begin{enumerate}
			\item \(M\) has finite injective dimension as an \(R\)-module.
			\item \(\Ext_R^i((R/pR)_{\perf}, M) = 0\) for all \(i \gg 0\).
			\item There exists a perfectoid \(R\)-algebra \(A\) satisfying \(\mfrakm A \neq A\) and \(\Ext^i_R(A, M) = 0\) for \(i \gg 0\).
		\end{enumerate}
	\end{theorem}
	
	Notably, this characterization is new in characteristic $p$, and applies not only to finitely generated \(R\)-modules but also to complete modules. Also, the satisfaction of any one of these three equivalent conditions implies that $R$ is Cohen--Macaulay.
	
	Our next result provides a characterization of Gorenstein local rings in terms of perfectoid algebras by using the Gorenstein projective dimension $\Gpd_R(-)$ introduced by Enochs and Jenda (\cite{enochs1995Gorenstein}) as an extension of Gorenstein dimension to not finitely generated modules.
	\begin{corollary} (Special case of \Cref{EquivGorCharp}).
		Let \((R, \mfrakm)\) be a Noetherian local ring with \(p \in \mfrakm\). Then the following are equivalent:
		\begin{enumerate}
			\item \(R\) is Gorenstein.
			\item The vanishing \(\Ext_R^i((R/pR)_{\perf}, R) = 0\) holds for \(i \gg 0\).  
			\item There exists a perfectoid \(R\)-algebra \(A\) such that \(\mfrakm A \neq A\) and \(\Gpd_R(A) < \infty\). 
		\end{enumerate}
	\end{corollary}
	
	In positive characteristic, Takahashi--Yoshino \cite[Theorem~6.2]{takahashi2004Characterizing} and also Iyengar--Sather-Wagstaff \cite[Theorem~A]{iyengar2004Gdimension} characterized Gorenstein local rings via the Gorenstein dimension of the Frobenius endomorphism.  Our results provide mixed-characteristic analogues of these characterizations, with the Frobenius endomorphism replaced by suitable perfectoid algebras.
	
	Using similar methods, the (balanced big) Cohen--Macaulayness of modules can be detected by the vanishing of \(\Ext\) in a suitable finite range, see~\Cref{TheoEquivCM} (\Cref{TheoEquivBCM}). This yields a characterization of a torsion-free module $M$ as a finite direct sum of the canonical module via the vanishing of $\Ext_R^i(A, M)$ for all $i \neq \dim(M)$; see \Cref{CorEquivcan}.
	

	In this introduction, we record the following result which is an injective-dimension analogue of finite-flat-dimensional criterion of Bhatt, Iyengar, and Ma \cite[Corollary 4.8 and Theorem 4.7]{bhatt2019Regular}.

	\begin{theorem}[\Cref{FinInjDimPerfRegular}]
		Let $(R,\mfrakm)$ be a Noetherian local ring with $p\in\mfrakm$. Then the following are equivalent:
		\begin{enumerate}
			\item $R$ is regular.
			\item There is a perfectoid $R$-algebra \(A\) such that $\idim_R(A)<\infty$ and \(\mfrakm A \neq A\).
			\item $R$ has an isolated singularity and there exists a perfectoid $R$-algebra $A$ such that \(\mfrakm A \neq A\) and \(\Ext_R^i(A,A)=0\) for all sufficiently large $i$.
		\end{enumerate}
	\end{theorem}
	We prove that, in several natural situations, the injective dimension of such an algebra 
	$A$ satisfies a dichotomy,
	$ 
	\idim_R(A)\in \{\dim(R), +\infty\},
	$ 
	(see \Cref{epa}). We further extend this to a more general setting, where the proof combines the preceding homological characterizations with techniques from the theory of solid algebras and non-vanishing results for the top local cohomology module 
	$H^d_{\mfrakm}(A)$ (see \Cref{embcm}).
	
	
	Some special cases of the present results appeared in \cite{asgharzadeh2026Perfect}. The current project, however, has a broader mixed-characteristic scope and settles an implicit conjecture from that paper---on the injective dimension of $R_{\perf}$. Due to substantial changes in content and structure, we present this as a separate paper.
	
	\subsection*{Acknowledgments}
Akiyoshi Sannai kindly suggested Remark \ref{aki1}, the error was identified by AI Math Referee, a mathematical auditing AI tool being developed by Akiyoshi Sannai.
	R.I. would like to thank Olgur Celikbas for suggesting some of the questions investigated in this paper. 
	R.I. is supported by JSPS KAKENHI Grant number 24KJ1085.
	\medskip
	\noindent
	
	\section{Projective and injective dimensions of perfectoid algebras}
	For a Noetherian local ring $(R,\mathfrak{m},k)$, let $\pd_R(-)$ and $\idim_R(-)$ denote the projective and injective dimensions, respectively. We refer to \cite{bhatt2019Regular} for the definition, foundational properties, and examples of perfectoid rings. In this section, we study modules of finite projective and injective dimension over perfectoid algebras and provide further characterizations of such modules. Applications are deferred to the next section.
	
	\begin{theorem} \label{SpectralSequence}
		Let $(R,\mfrakm,k)$ be a local ring, and let \(S\) be a nonzero \(R\)-algebra with a proper ideal \(J\) such that \(\mfrakm S \subseteq J\) and \(d \defeq \pd_S(S/J) < \infty\). Choose an integer \(s \in \setZ\).
		If an \(R\)-module \(M\) satisfies \(\Ext_R^i(S, M) = 0\) for \(i = s, \dots, s+d\), then \(\Ext_{R}^{s+d}(k, M) = 0\).
	\end{theorem}
	
	\begin{proof}
		The proof is almost the same as that of \cite[Theorem 2.1]{bhatt2019Regular}. Consider the spectral sequence (see, for example, \citeSta{0AVG})
		\[
		E_2^{p,q} = \Ext_S^p(S/J, \Ext_R^q(S, M)) \Rightarrow \Ext_R^{p+q}(S/J, M).
		\]
		Since \(\Ext_R^q(S, M) = 0\) for \(q = s, \dots, s+d\) and \(\Ext_R^p(S/J, -) = 0\) for \(p > d\), we have \(E_2^{p,q} = 0\) for all \(p, q \in \mathbb Z\) with \(p + q = s + d\). This shows that \(\Ext_R^{s+d}(S/J, M) = 0\).
		Since \(J\) contains \(\mfrakm S\), the ring \(S/J\) is a nonzero \(R/\mfrakm\)-module, and hence there is an isomorphism \(S/J \cong k^{\oplus \Lambda}\) for some nonempty index set \(\Lambda\). Since the functor \(\Ext_R^{s+d}(-, M)\) sends direct sums to products, we have
		\[
		\Ext_R^{s+d}(S/J, M) \cong \prod_{\Lambda} \Ext_R^{s+d}(k, M).
		\]
		Therefore, we conclude the desired vanishing \(\Ext_R^{s+d}(k, M) = 0\).
	\end{proof}
	
	We will show that the projective dimension of a ``perfectly principal ideal'' is at most one. 
	
	\begin{lemma} \label{LemFinPdim}
		Let \(R\) be a reduced ring and let \(x\) be an element of \(R\) with a compatible system of \(p\)-power roots \(\{x^{1/p^n}\}_{n \geq 1}\) in \(R\). Then the ideal \((x^{1/p^\infty}) = \bigcup_{n \geq 1} x^{1/p^n}R\) satisfies \(\pd_R((x^{1/p^\infty})) \leq 1\).
	\end{lemma}
	
	\begin{proof}
		As in the proof of \cite[Lemma 3.2]{bhatt2019Regular}, the \(R\)-module \((x^{1/p^\infty})\) can be identified with the colimit
		\[
		(x^{1/p^\infty}) \cong \colim_{n \ge 0}
		\left(
		R \xrightarrow{\times x^{1/p - 1/p^2}} R \xrightarrow{\times x^{1/p^2 - 1/p^3}} \cdots
		\xrightarrow{\times x^{1/p^n - 1/p^{n+1}}} R \xrightarrow{\times x^{1/p^{n+1} - 1/p^{n+2}}} \cdots
		\right),
		\]
		since \(R\) is reduced. For any \(R\)-module \(M\), we have
		\[
		R\Hom_R((x^{1/p^\infty}), M) \cong R\lim_{n \ge 0} \bigl( R\Hom_R(R, M) \bigr) \cong R\lim_{n \ge 0} M
		\]
		in \(\mathcal D(R)\), where the transition map \(M \to M\) is given by multiplication by \(x^{1/p^n - 1/p^{n+1}}\).
		Since we have a distinguished triangle
		\[
		R\lim_{n \ge 0} M \longrightarrow \prod_{n \ge 0} M \xrightarrow{1 - \mathrm{shift}} \prod_{n \ge 0} M
		\]
		in \(\mathcal D(R)\), the cohomology groups \(H^i(R\Hom_R((x^{1/p^\infty}), M))\) vanish for \(i \neq 0, 1\). This shows that the projective dimension of \((x^{1/p^\infty})\) is at most one as an \(R\)-module.
	\end{proof}
	
	\begin{proposition} \label{TheoFinPdim}
		Let \(R\) be a ring and let \(I\) be an ideal of \(R\) such that \(I\) contains \(p\) and is generated by \(r\) elements up to radical. Let \(A\) be a reduced \(R\)-algebra such that it contains a compatible system \(\{f_0^{1/p^n}\}_{n \geq 1}\) of \(p\)-power roots of a unit multiple of \(p\) and \(A/(f_0^{1/p^\infty})\) is perfect.\footnote{By \cite[Lemma 3.9]{bhatt2018Integral}, any perfectoid \(R\)-algebra satisfies these conditions.}
		Then \(A/\sqrt{IA}\) has a finite projective dimension as an \(A\)-module and, in particular, \(\pd_A(A/\sqrt{IA}) \leq 2(r+1)\).
	\end{proposition}
	
	\begin{proof}
		Let \(\sqrt{I} = \sqrt{(f_1, \dots, f_r)} = \sqrt{(p, f_1, \dots, f_r)}\) be a generating set of \(I\) up to radical in \(R\). This implies that \(\sqrt{IA} = \sqrt{(f_0, f_1, \dots, f_r)A}\) in \(A\).
		Since \(A/(f_0^{1/p^\infty})\) is perfect, the class \(\overline{f_i} \in A/(f_0^{1/p^\infty})\) has a compatible system of \(p\)-power roots \(\{\overline{f_i}^{1/p^n}\}_{n \geq 1}\) in \(A/(f_0^{1/p^\infty})\) for each \(i = 1, \dots, r\).
		Writing these representatives in \(A\) as \(f_{i, n}\) in \(A\) for \(n \geq 0\), we have \(f_{i, 0} = f_i\) and \(f_{i, n+1}^p - f_{i, n} \in (f_0^{1/p^\infty})\) for any \(n \geq 0\).
		Then we have \(\sqrt{IA} = (f_0^{1/p^n}, f_{1, n}, \dots, f_{r, n})_{n \geq 0}\) since \(A/(f_0^{1/p^\infty})\) is perfect.
		For any \(A\)-module \(M\), we have a distinguished triangle
		\begin{equation*}
			R\Hom_A(A/(f_0^{1/p^\infty}), M) \to R\Hom_A(A, M) \to R\Hom_A((f_0^{1/p^\infty}), M)
		\end{equation*}
		in \(\mcalD(A)\). By \Cref{LemFinPdim}, the ideal \((f_0^{1/p^\infty})\) has a projective dimension at most one as an \(A\)-module since \(A\) is reduced. This shows that the first term \(R\Hom_A(A/(f_0^{1/p^\infty}), M)\) is concentrated in cohomological degree \([0, 2]\).
		Next, we have a distinguished triangle
		\begin{equation*}
			R\Hom_A(A/(f_0^{1/p^\infty}, f_{1, n})_{n \geq 0}, M) \to R\Hom_A(A/(f_0^{1/p^\infty}), M) \to R\Hom_A((\overline{f_1}^{1/p^\infty})A/(f_0^{1/p^\infty}), M)
		\end{equation*}
		in \(\mcalD(A)\). The middle term is concentrated in cohomological degree \([0, 2]\) as we have seen above. Since \(A/(f_0^{1/p^\infty})\) satisfies the assumption on \Cref{LemFinPdim}, we have
		\begin{equation*}
			R\Hom_A((\overline{f_1}^{1/p^\infty})A/(f_0^{1/p^\infty}), M) \cong R\lim_{n \geq 0}R\Hom_A(A/(f_0^{1/p^\infty}), M)
		\end{equation*}
		in \(\mcalD(A)\) as in the proof of the lemma. The right-hand side is the derived limit of a system of complexes concentrated in cohomological degree \([0, 2]\), and hence it is concentrated in cohomological degree \([0, 3]\). This shows that the first term \(R\Hom_A(A/(f_0^{1/p^\infty}, f_{1, n})_{n \geq 0}, M)\) is concentrated in cohomological degree \([0, 4]\). By repeating the same argument, we see that the complex \(R\Hom_A(A/\sqrt{IA}, M) = R\Hom_A(A/(f_0^{1/p^\infty}, f_{1, n}, f_{2, n}, \dots, f_{r, n})_{n \geq 0}, M)\) is concentrated in cohomological degree \([0, 2(r+1)]\). So the projective dimension of \(A/\sqrt{IA}\) is at most \(2(r+1)\) as an \(A\)-module.
	\end{proof}
	
	\begin{remark} \label{RemarkProjDim}
		We review some previous works on the projective dimension of perfect(oid) algebras related to the above two results.
		\begin{enumerate}
			\item In \cite[Theorem 3.1]{aberbach1997Finite}, they consider a perfect ring \(R\) of characteristic \(p\) and \(p\)-torsion-free reduced ring \(R\) such that \(\sqrt{pR}\) is generated by roots of \(p\) and \(R/\sqrt{pR}\) is perfect. Under this assumption, they proved that the flat dimension holds \(\fdim_R (R/I) \leq r\) if an ideal \(I\) is \(\sqrt{(x_1, \dots, x_r)}R\) and \(x_1 = p\) in the second case.
			\item After that, in \cite[Lemma 3.2]{bhatt2019Regular}, they gave an elementary proof of the above result for perfect rings of characteristic \(p\). As explained in the proof \Cref{LemFinPdim}, we follow their argument.
			\item Similarly but independently, in \cite[Lemma 4.4]{asgharzadeh2010Homological}, the first-named author proved that if \(R\) is an integral domain of characteristic \(p\), then the ideal \((x^{\infty})R^+\) generated by roots of \(x\) in the absolute integral closure \(R^+\) has a projective dimension one as an \(R^+\)-module. Using this, the upper bound of the projective dimension of \(R^+/(x_1^{\infty}, \dots, x_r^{\infty})R^+\) is also given in \cite[Lemma 4.5]{asgharzadeh2010Homological}~ as Aberbach--Hochster's result.
		\end{enumerate}
		Consequently, \Cref{LemFinPdim} and \Cref{TheoFinPdim} are slight modifications of the above results in the setting of projective dimension.
	\end{remark}
	
	Similarly, we can find an upper bound of the injective dimension of quotients of perfectoid algebras if its injective dimension is finite.
	
	\begin{proposition} \label{TheoFinIdim}
		Keep the setting of  \Cref{TheoFinPdim}. Assume further that \(R\) is Noetherian and \(A\) has finite injective dimension as an \(R\)-module. Then \(A/\sqrt{IA}\) has a finite injective dimension as an \(R\)-module.
		In particular, one has \(\idim_R(A/\sqrt{IA}) \leq \idim_R (A)\).
	\end{proposition}
	
	\begin{proof}
		We take \(f_{1, n}, \dots, f_{r, n}\) as in the proof of \Cref{TheoFinPdim}. Let \(M\) be a finitely generated \(R\)-module.
		We have a distinguished triangle
		\[
		R\Hom_R(M, (f_0^{1/p^\infty})A) \to R\Hom_R(M, A) \to R\Hom_R(M, A/(f_0^{1/p^\infty}))
		\]
		in \(\mathcal D(R)\). The middle term \(R\Hom_R(M, A)\) is concentrated in cohomological degrees \([0, \idim_R(A)]\). Since \(M\) is finitely generated and \(R\) is Noetherian, we have
		\[
		R\Hom_R(M, (f_0^{1/p^\infty})A) \cong \colim_{n \ge 0} R\Hom_R(M, A)\quad(\ast),
		\]
		in \(\mathcal D(R)\) (see \citeSta{0G8V}), where the transition map is given by multiplication by \(f_0^{1/p^n - 1/p^{n+1}}\).
		Since \(R\Hom_R(M, A)\in\mathcal D^{[0, \idim_R(A)]}(R)\), the filtered colimit $(\ast)$ also belongs to \(\mathcal D^{[0, \idim_R(A)]}(R)\). Hence, the last term \(R\Hom_R(M, A/(f_0^{1/p^\infty}))\) is also concentrated in cohomological degrees \([0, \idim_R(A)]\). This shows that \(A/(f_0^{1/p^\infty})\) has finite injective dimension at most \(\idim_R(A)\) as an \(R\)-module (see, for example, \citeSta{0A5T}).
		Next, we have a distinguished triangle
		\[
		R\Hom_R(M, (\overline{f_1}^{1/p^\infty})A/(f_0^{1/p^\infty})) 
		\to R\Hom_R(M, A/(f_0^{1/p^\infty})) 
		\to R\Hom_R(M, A/(f_0^{1/p^\infty}, f_{1, n})_{n \ge 0})
		\]
		in \(\mathcal D(R)\). The middle term is concentrated in cohomological degrees \([0, \idim_R(A)]\) as shown above. Moreover, the same argument applied to the perfect \(R\)-algebra \(A/(f_0^{1/p^\infty})\)—which has injective dimension at most \(\idim_R(A)\)—shows that the first term 
		\(R\Hom_R(M, (\overline{f_1}^{1/p^\infty})A/(f_0^{1/p^\infty}))\) is also concentrated in cohomological degrees \([0, \idim_R(A)]\). Consequently, the third term satisfies the same bound.
		By repeating this argument, we obtain that 
		\[
		A/(f_0^{1/p^\infty}, f_{1, n}, \dots, f_{r, n})_{n \ge 0} = A/\sqrt{IA}
		\]
		has finite injective dimension at most \(\idim_R(A)\) as an \(R\)-module.
	\end{proof}
	
	
	
	Unless otherwise stated, by a complete module we mean complete with respect to the 
	$\mfrakm$-adic topology.
	While we assume that \(R\) has isolated singularity in (2) of the following corollary, this assumption can be removed if \(M\) is finitely generated or even complete (see \Cref{TheoEquivGor} (2) below).
	
	\begin{corollary} \label{CorVanishing}
		Let \((R,\mfrakm, k)\) be a $d$-dimensional Noetherian local ring and assume \(p \in \mfrakm\).  
		Let \(A\) be a reduced \(R\)-algebra satisfying \(\mfrakm A \neq A\), and suppose \(A\) is equipped with a compatible system \(\{f_0^{1/p^n}\}_{n \ge 1}\) of \(p\)-power roots of a unit multiple of \(p\). Assume further that \(A/(f_0^{1/p^\infty})\) is perfect.
		
		If an \(R\)-module \(M\) satisfies
		\( 
		\Ext_R^i(A, M) = 0\)   {for} \( i = s, \dots, s + 2(d+ 1)
		\)
		for some integer \(s \in \mathbb Z\), then
		$    
		\Ext_R^{s + 2(d + 1)}(k, M) = 0.
		$    
		In particular, the following assertions are valid.
		\begin{enumerate}
			\item If \(s \ge -d - 2\), then
			\(
			\Ext_R^i(k, M) = 0 \)  for all   \(i \ge s + 2(d+ 1) \).
			\item If, in addition to (1), \(R\) has an isolated singularity, then \(M\) has finite injective dimension as an \(R\)-module.
		\end{enumerate}
	\end{corollary}
	
	\begin{proof}
		We first prove the main vanishing assertion. Set \(D \coloneqq 2(d+ 1)=2(\dim R + 1)\) and \(e \coloneqq \pd_A(A/\sqrt{\mfrakm A})\). By \Cref{TheoFinPdim}, we have \(e \le D\), since \(\mfrakm\) can be generated, up to radical, by at most \(\dim R\) elements. The quotient \(A/\sqrt{\mfrakm A}\) is a nonzero \(k\)-algebra because \(\mfrakm A \neq A\).
		The assumed vanishing gives
		\[
		\Ext_R^i(A, M) = 0 \quad \text{for } i = s + D - e, \dots, s + D.
		\]
		Applying \Cref{SpectralSequence} to the data \((R, A, \sqrt{\mfrakm A}, M, e)\), we obtain
		\[
		\Ext_R^{s + D}(k, M)
		= \Ext_R^{s + D - e + e}(k, M)
		= 0,
		\]
		which proves the first assertion.
		
		For the final assertion, note that it follows from \cite[Theorem 1.2]{christensen2019Rigidity}, once we establish that \(\Ext_R^i(k, M) = 0\) for some \(i \ge d\), which is precisely the content of part (1) under the given hypotheses.
		Combining  $(1)$ with \cite[Theorem VI.9]{schoutens2003Projectivea} yields the desired conclusion (2).
	\end{proof}
	
	\medskip
	\noindent
	
	\section{Characterization via perfectoid algebras}
	
	Building on the results from the previous section, we now give a characterization of homological dimensions of finitely generated modules over Noetherian local rings in terms of perfectoid algebras. In order to handle the analogous perfectoid conditions that will appear below, we introduce the following notation.
	
	\begin{notation} \label{GeneralPerfectoidCondition}
		Let \((R, \mfrakm, k)\) be a Noetherian local ring with \(p \in \mfrakm\) and let \(A\) be an \(R\)-algebra such that \(\mfrakm A \neq A\).
		Consider the following   conditions on \(A\):
		\begin{enumalphp}
			\item \(A\) is a perfectoid \(R\)-algebra. \label{ConditionPerfd}
			\item \(A\) is a reduced \(R\)-algebra such that it contains a compatible system \(\{f_0^{1/p^n}\}_{n \geq 1}\) of \(p\)-power roots of a unit multiple of \(p\) and \(A/(f_0^{1/p^\infty})\) is perfect. \label{ConditionGeneralPerfd}
		\end{enumalphp}
		Since the implication \ref{ConditionPerfd} \(\Rightarrow\) \ref{ConditionGeneralPerfd} holds by \cite[Lemma 3.9]{bhatt2018Integral}, we will often only consider the condition \ref{ConditionGeneralPerfd} in the following results.
	\end{notation}
	
	\subsection{Detection of finite injective dimension and Gorenstein property}
	
	First, we consider the injective dimension, and present three consequences.
	
	\begin{theorem} \label{TheoEquivGor}
		Let \((R, \mfrakm, k)\) be a Noetherian  local ring with \(p \in \mfrakm\) and let \(M\) be a nonzero finitely generated or   \(\mfrakm\)-adically complete \(R\)-module. Then the following are equivalent:
		\begin{enumerate}
			\item \(M\) has finite injective dimension as an \(R\)-module.
			\item The vanishing \(\Ext_R^i((R/pR)_{\perf}, M) = 0\) holds for \(i \gg 0\).
			\item The vanishing \(\Ext^i_R(A, M) = 0\) holds for \(i \gg 0\) for some \(R\)-algebra \(A\) satisfying one of the conditions in  \Cref{GeneralPerfectoidCondition}.
		\end{enumerate}In particular, the satisfaction of any one of these equivalent conditions implies that $R$ is Cohen--Macaulay.
	\end{theorem}
	
	\begin{proof}
		(1) \(\Rightarrow\) (2) \(\Rightarrow\) (3): Since \(M\) is an \(R\)-module with finite injective dimension, we have \(\Ext_R^i(A, M) = 0\) for \(i \gg 0\) and for any \(R\)-algebra \(A\).
		
		(3) \(\Rightarrow\) (1): By \citeSta{0AVJ}~ for finitely generated modules or \cite[\S 3.3]{simon1990Homological} for \(\mfrakm\)-adically complete modules, it suffices to show that \(\Ext_R^i(k, M) = 0\) for \(i \gg 0\) when (3) holds for \(A\) satisfying the condition \Cref{GeneralPerfectoidCondition} \ref{ConditionGeneralPerfd}. By \Cref{CorVanishing} and \(\mfrakm A \neq A\), it suffices to show that \(\Ext_R^i(A, M) = 0\) for \(i \gg 0\), which holds by the assumption.
		
		The last assertion follows as an immediate consequence of \cite[Theorem~7.5]{simon1990Homological} (resp. \cite[Corollary 9.6.2]{bh}) when $M$ is complete (resp. finitely generated).
	\end{proof}

	The following may be compared with \cite[Theorem 4.4]{bhatt2019Regular}; however, unlike there, the module here is $\mfrakm$-torsion (for example Artinian) and not necessarily finitely generated.
	
	\begin{corollary}
		Let $(R,\mfrakm,k)$ be a Noetherian local ring with \(p \in \mfrakm\), and let
		$L$ be an $\mfrakm$-torsion $R$-module. Suppose that \(\Tor_i^R(A, L) = 0\) for \(i \gg 0\), where $A$ is an $R$-algebra satisfying one of the conditions in  \Cref{GeneralPerfectoidCondition}.
		Then \(\pd_R(L) < \infty\), and consequently $R$ is Cohen--Macaulay.
	\end{corollary}
	
	\begin{proof}
		Let $(-)^\vee \defeq \Hom_R(-,E_R(k))$ denote the Matlis duality functor.
		By \cite[\S 4.1]{simon1990Homological}, one has
		$0= \Tor_i^R(A,L)^\vee
		\cong
		\Ext_R^i(A,L^\vee).$	 
		Since $L$ is $ \mfrakm$-torsion, and by \cite[\S 4.2, Lemma]{simon1990Homological} we deduce that $L^\vee$ is complete.
		Hence, \Cref{TheoEquivGor} shows $\idim_R(L^\vee)<\infty$.
		By using the same formula, for any \(R\)-module \(N\) and any \(i > \idim_R(L^{\vee})\), we have
		\begin{equation*}
			\Tor_i^R(N, L)^{\vee} \cong \Ext_R^i(N, L^{\vee}) \cong 0.
		\end{equation*}
		Since the Matlis duality functor is faithful, we conclude that \(\Tor_i^R(N, L) = 0\) for all \(i > \idim_R(L^{\vee})\). In particular, we have \(\pd_R(L) < \infty\).
		Finally, the existence of a nonzero $\mfrakm$-torsion  module of finite projective dimension implies that $R$ is Cohen--Macaulay (\cite[Observation 3.15]{asgharzadeh2023Notes}).
	\end{proof}
	
	Since the perfectoid algebras can be regarded as a counterpart of the Frobenius endomorphism (more precisely, the perfect closure) in mixed characteristic, it is natural to expect that the Gorenstein property of Noetherian local rings can be characterized in terms of perfectoid algebras as done in positive characteristic (\cite{herzog1974Ringe,takahashi2004Characterizing,iyengar2004Gdimension})
	This is the following corollary.
	
	\begin{corollary} \label{EquivGorCharp}
		Let \((R, \mfrakm, k)\) be a Noetherian local ring with \(p \in \mfrakm\). Then the following are equivalent:
		\begin{enumerate}
			\item \(R\) is Gorenstein.
			\item The vanishing \(\Ext_R^i((R/pR)_{\perf}, R) = 0\) holds for \(i \gg 0\).
			\item There exists \(n > 0\) such that for infinitely many \(e > 0\), the vanishing \(\Ext_R^i(F_\ast^e(R/pR), R) = 0\) holds for \(i \geq n\).
			\item The vanishing \(\Ext_R^i(A, R) = 0\) holds for \(i \gg 0\) for some \(R\)-algebra \(A\) satisfying one of the conditions in  \Cref{GeneralPerfectoidCondition}.
			\item There exists an \(R\)-algebra \(A\) satisfying one of the conditions in  \Cref{GeneralPerfectoidCondition} such that its Gorenstein projective dimension is finite, i.e., \(\Gpd_R(A) < \infty\).\footnote{For the notion of Gorenstein projective/flat dimensions, we use \cite{christensen2011Totally} as a reference. Following this reference, we will cite the results on Gorenstein projective/flat dimensions.}
		\end{enumerate}
	\end{corollary}
	
	\begin{proof}
		The equivalences (1)~$\Leftrightarrow$~(2)~$\Leftrightarrow$~(4) are established in \Cref{TheoEquivGor}.
		
		If $R$ is Gorenstein, then every $R$-module has finite Gorenstein projective dimension (\cite[Theorem 12.3.1]{enochs2000Relative} or \cite[Theorem 2.19]{christensen2011Totally}); hence (1)~$\Rightarrow$~(5) holds. We now show (5)~$\Rightarrow$~(4). Assume that $\Gpd_R(A) < \infty$ for some $R$-algebra $A$ satisfying condition~\ref{ConditionGeneralPerfd} in \Cref{GeneralPerfectoidCondition}. Then, by \cite[Theorem 2.20]{holm2004Gorenstein}, we have $\Ext_R^i(A, R) = 0$ for all $i > \Gpd_R(A)$.
		
		The implication (1)~$\Rightarrow$~(3) is trivial, since $R$ has finite injective dimension as it is a Gorenstein ring. It remains to prove (3)~$\Rightarrow$~(2). We have an isomorphism
		\begin{equation*}
			R\Hom_R((R/pR)_{\perf}, R) \cong R\lim_{e \geq 0}R\Hom_R(F_*^e(R/pR), R) \cong R\lim_{e' \in E} R\Hom_R(F_*^{e'}(R/pR), R)
		\end{equation*}
		in $\mcalD(R)$, where $E \subseteq \mathbb{Z}_{\geq 0}$ is an infinite subset such that $\Ext_R^i(F_*^{e'}(R/pR), R) = 0$ for all $i \geq n$ and $e' \in E$. Writing $M_{e'}  \defeq  R\Hom_R(F_*^{e'}(R/pR), R)$, we have a distinguished triangle
		\begin{equation*}
			R\lim_{e' \in E}M_{e'} \to \prod_{e' \in E}M_{e'} \xrightarrow{1 - \text{shift}} \prod_{e' \in E}M_{e'}
		\end{equation*}
		in $\mcalD(R)$. Since each $M_{e'}$ is concentrated in cohomological degrees $[0, n-1]$, the triangle shows that $R\Hom_R((R/pR)_{\perf}, R)$ is concentrated in degrees $[0, n]$. In particular, $\Ext_R^i((R/pR)_{\perf}, R) $ is zero for all sufficiently large $i$. This proves (3)~$\Rightarrow$~(2), completing the proof.
	\end{proof}
	
	
	Using the above results, we can show that the vanishing of \(\Ext\) modules in all sufficiently high degrees implies the vanishing of all \(\Ext\) modules in positive degrees.
	
	\begin{corollary} \label{ExtAllVanishing}
		Let \((R, \mfrakm, k)\) be a complete Noetherian local ring of characteristic \(p\), and let \(M \neq 0\) be  a finitely generated \(R\)-module. Then the following are equivalent:
		\begin{enumerate}
			\item \(\Ext^i_R(R_{\perf},M) = 0\) for all \(i > 0\);
			\item \(\Ext^i_R(R_{\perf},M) = 0\) for all \(i \gg 0\).
		\end{enumerate}
		In particular, $R$ is Gorenstein if and only if one of the equivalent conditions:
		\begin{equation*}
			\Ext^{>0}_R(R_{\perf}, R) = 0\Leftrightarrow\Ext^{\gg 0}_R(R_{\perf}, R) = 0\Leftrightarrow\Gpd_R(R_{\perf}) < \infty.
		\end{equation*}
	\end{corollary}
	
	\begin{proof}
		Suppose \(\Ext^i_R(R_{\perf}, M) = 0\) for all \(i \gg 0\). By \Cref{TheoEquivGor}, \(M\) has finite injective dimension, and so \(R\) is Cohen--Macaulay (\cite[Corollary 9.6.2]{bh}). It follows that \(R_{\perf}\) is a balanced big Cohen--Macaulay \(R\)-algebra. Then \cite[Corollary~7.7]{simon1990Homological} implies that \(\Ext^i_R(R_{\perf}, M) = 0\) for all \(i > 0\).
		
		Applying the first part to $M  \defeq  R$ together with \Cref{EquivGorCharp} now yields the last assertion immediately.
	\end{proof}

	\begin{remark}\label{aki1}
		Akiyoshi Sannai informed us that the completeness assumption in
		\Cref{ExtAllVanishing} is essential. We thank him for providing the
		following example, which shows that the vanishing of
		$\Ext^i_R(R^{\perf},R)$ for all sufficiently large $i$ does not, in
		general, imply its vanishing for every $i>0$, even when $R$ is
		Gorenstein.
        
        To this end, let $k:=\mathbb{F}_p(t_1,t_2,\ldots)$
		be the rational function field in countably many algebraically
		independent variables over $\mathbb{F}_p$. Consider the subring
		\[
		R:=k^p[[x]][k]\subseteq k[[x]],
		\]
		generated by $k^p[[x]]$ and $k$, and let $K$ denote its fraction
		field. In 	\Cref{ExtAllVanishing} we take $M=R$. The ring $R$ is a classical example of a
		discrete valuation ring that is not excellent. We first record
		some properties of this ring that will be used below.
		
		Every element of $R$ belongs to $L[[x]]$ for some finite
		intermediate field $k^p\subseteq L\subseteq k$.
        If $L$ is such a field, then $L/k^p$ is finite and $L[[x]]$ is a finite
		free $k^p[[x]]$-module. Conversely, each element of $L[[x]]$ is a
		$k^p[[x]]$-linear combination of a $k^p$-basis of $L$, so
		\[
		R=\bigcup_{k^p\subseteq L\subseteq k,\,[L:k^p]<\infty}L[[x]].
		\]
		Since every $L[[x]]$ is a discrete valuation ring with uniformizer
		$x$, every nonzero element of $R$ has the form $x^n v$ for some
		$n\geq 0$ and some unit $v\in R$. Thus $R$ is a discrete valuation
		ring with uniformizer $x$, and in particular it is a regular local
		ring.
		
        We next identify its completion. For every finite intermediate field $L$ and every $n\geq 0$, we have $x^nR\cap L[[x]]=x^nL[[x]]$ since \(L[[x]] \subseteq R\) is an extension of discrete valuation rings. Moreover, $K\cap k[[x]]=R$.
        It follows that
        $	R/x^nR\cong k[x]/(x^n)$
        for every $n\geq 1$, and hence
        $\widehat R\cong k[[x]].$

        We next identify the relation between fraction fields of \(k[[x]]\) and \(R\).
        Since every Laurent series has only finitely many terms with negative powers of $x$, any element of \(k((x))\) can be written as $f=\sum_{i=-n}^{-1}a_i x^i+g(x)$ for some $n\geq 0$, where $a_i\in k$ and $g(x)\in k[[x]]$. Since $k\subseteq R\subseteq K$ and \(x\in R\), the finite negative power part $\sum_{i=-n}^{-1}a_i x^i$ belongs to \(K\).
        Hence $f\in K+k[[x]]$, proving that $k((x))\subseteq K+k[[x]]$. Conversely, $K\subseteq k((x))$ because $R\subseteq k[[x]]$, and therefore $K+k[[x]]\subseteq k((x))$. We conclude that
        $k((x))=K+k[[x]]$.
        
        Now, we can pick the power series $u=\sum_{i\geq 1}t_i x^i\in k[[x]].$
        Then $u\notin R$. Otherwise, $u$ would belong to $L[[x]]$ for
        some finite extension $L/k^p$, forcing $t_i\in L$ for every
        $i\geq 1$. This is impossible because
        $[k^p(t_1,\ldots,t_n):k^p]=p^n$
		for every $n\geq 1$. Since $K\cap k[[x]]=R$, we also have
		$u\notin K$. On the other hand, as the characteristic is $p$, applying Frobenius map gives us the following equation
        $$u^p=\sum_{i\geq 1}t_i^p x^{ip}\in k^p[[x]]\subseteq R.$$	In particular, $u$ belongs to $R^{1/p}$, where the latter is
        viewed inside an algebraic closure of $K$.
		
		We now verify the vanishing condition
		\Cref{ExtAllVanishing}(2). Since $R$ is a discrete valuation ring, it is Gorenstein and $\Ext^i_R(R^{\perf},R)=0$ for every $i\geq 2,$ and $\Gpd_R(R_{\perf})\leq 1$ hold.
    	We claim, however, that \Cref{ExtAllVanishing}(1) fails. More
		precisely, we will construct an $R$-submodule
		$N\subseteq R^{\perf}$ such that
	    $\Ext^1_R(N,R)\neq 0.$
	    This completes the proof: Indeed, since $R$ has global dimension \(1\), $\Ext^2_R(\Rperf/N,R)=0$ holds and the long exact sequence induces from $0\to N \to \Rperf \to\Rperf/N\to 0$ to deduce the following exact sequence
        $$\Ext^1_R(\Rperf,R)\to \Ext^1_R(N,R)  \to \Ext^2_R(\Rperf /N,R)=0.$$
        Now if \(\Ext^1_R(\Rperf,R) \) is zero, then we should have
        $\Ext^1_R(N,R)=0$, a contradiction.

\medskip
\noindent\textit{Construction of the module $N$.}
Set
$N \defeq (K+Ku)\cap k[[x]].$
This is an $R$-submodule of $k[[x]]$. We first show that
$N\subseteq R^{1/p}$. If $a,b\in K$ and $a+bu\in k[[x]]$, then
$(a+bu)^p=a^p+b^pu^p\in K\cap k[[x]]=R.$
Thus $a+bu$ is a $p$-th root of an element of $R$, and consequently
$N\subseteq R^{1/p}\subseteq R^{\perf}.$
		
Since $u\notin K$, the elements $1$ and $u$ are linearly
independent over $K$. Hence the assignment
$$\varphi:N\longrightarrow K,\qquad a+bu\longmapsto b$$
is well-defined and $R$-linear. Its kernel is
$K\cap k[[x]]=R$. To see that $\varphi$ is surjective, let
$b\in K$. By the equality $k((x))=K+k[[x]]$ as proved above, we can write
$bu=c+d$ for some $c\in K$ and $d\in k[[x]]$. Taking $a \defeq -c$
gives $a+bu=d\in N$ and $\varphi(a+bu)=b$. We therefore obtain
an exact sequence
$$(\ast):\quad		0\longrightarrow R\longrightarrow N
\stackrel{\varphi}\longrightarrow K\longrightarrow 0.$$

Here, we show the extension of \(R\)-modules does not split. 
Suppose, to the contrary, that $(*)$ splits. Then there is an
$R$-linear section $s:K\to N$ of $\varphi$. Its image
$D \defeq s(K) \subseteq N$ is a nonzero submodule of $N$ isomorphic to $K$.
Because multiplication by the uniformizer $x$ of \(R\) is surjective on the fraction field $K$, we have
$xD=D$. Iterating this equality yields
$$	D=x^nD\subseteq x^nk[[x]]
\qquad\text{for every }n\geq 0.
$$	But $  \bigcap_{n\geq 0}x^nk[[x]]=0,  $ so $D=0$, a contradiction. Thus $(*)$ is nonsplit.

Since this non-splitting extension \((\ast)\) defines a non-zero element of \(\Ext^1_R(K, R)\), we conclude that \(\Ext^1_R(N, R) \neq 0\) as desired.
In summary, we observed that \Cref{ExtAllVanishing}(1) does not follow from \Cref{ExtAllVanishing}(2).
 	
		
Finally, the same example also disproves the ``In particular'' assertion of \Cref{ExtAllVanishing} when we do not assume \(R\) is complete:
Indeed, the constructed $R$ is Gorenstein, $\Ext_R^i(R_{\perf},R)=0 $ for all  $i\geq 2$, and $\Gpd_R(R_{\perf})<\infty$, whereas $\Ext_R^1(R_{\perf},R)\neq 0$.
Thus the first equivalence in the display fails, as does the implication from the Gorenstein property of $R$ to the vanishing of $\Ext_R^{>0}(R_{\perf},R)$.
\end{remark}

	We record the following result which is a dual form of \Cref{ExtAllVanishing} above.
	
	\begin{corollary} \label{TorAllVanishing}
		Let \((R, \mfrakm, k)\) be an excellent local domain with \(p \in \mfrakm\), and let \(M\) be a finitely generated \(R\)-module. Let $A$ be the p-adic completion of the absolute integral closure of $R$. Then the following are equivalent:
		\begin{enumerate}
			\item \(\Tor^R_i(A, M) = 0\) for all \(i > 0\);
			\item \(\Tor^R_i(A, M) = 0\) for all \(i \gg 0\).
		\end{enumerate}
	\end{corollary}
	
	\begin{proof}
		Suppose \(\Tor^R_i(A, M) = 0\) for all \(i \gg 0\). By \cite[Theorem 4.4]{bhatt2019Regular}, \(\pd_R(M)\) is finite. Recall that \(A\) is a balanced big Cohen--Macaulay \(R\)-algebra by \cite[Main Theorem 5.15]{hochster1992Infinite}, \cite[Corollary 5.17]{bhatt2021CohenMacaulayness}, and \cite[Corollary 2.10]{bhatt2023Globally}. Then \cite[Lemma 4.1]{asgharzadeh2021Note}\footnote{The argument in the cited lemma requires \(N\) to be balanced big Cohen--Macaulay. Indeed, although the \(\mfrakm\)-adic completion \(\widehat{N}\) of a big Cohen--Macaulay module is balanced, completion is not faithful for arbitrary non-finitely generated modules, so one cannot recover the vanishing of \(\Tor_i^R(M,N)\) from that of \(\Tor_i^R(M,\widehat{N})\).} implies (1).
	\end{proof}
	
	\subsection{Characterization of Cohen--Macaulay modules}
	Since we have the precise degree of vanishing of \(\Ext\) modules in \Cref{CorVanishing}, we can also give a characterization of Cohen--Macaulay modules in terms of perfectoid algebras.

	\begin{theorem} \label{TheoEquivCM}
		Let \((R, \mfrakm, k)\) be a Noetherian local ring with \(p \in \mfrakm\), and let \(M\) be  nonzero and finitely generated as an $R$-module.
		Then the following are equivalent:
		\begin{enumerate}
			\item \(M\) is Cohen--Macaulay.
			\item There exists an \(R\)-algebra \(A\) satisfying one of the conditions in \Cref{GeneralPerfectoidCondition} such that \(\Ext_R^i(A, M) = 0\) for all \(i < \dim(M)\).
		\end{enumerate}
	\end{theorem}
	
	\begin{proof}
		Recall that \(k_{\perf}\)    is perfectoid ring of characteristic $p$ is precisely perfect ring of characteristic $p$. Since $p$ is zero in   \(k_{\perf}\) we can take $f_0$ as $0$. Then the quotient \(k_{\perf}=k_{\perf}/0\)  is a perfect ring. So, this  satisfies conditions in \Cref{GeneralPerfectoidCondition}.
		
		(1) \(\Rightarrow\) (2): The perfection \(k_{\perf}\) if the residue field is a desired perfectoid \(R\)-algebra. Indeed, if \(M\) is Cohen--Macaulay, then \(\Ext_R^i(k, M) = 0\) for all \(i < \depth(M) = \dim(M)\). Since \(k_{\perf} \cong \bigoplus_{\Lambda} k\), we have
		\( 
		\Ext_R^i( k_{\perf}, M) \cong \prod_{\Lambda} \Ext_R^i(k, M) = 0
		\)
		for all \(i < \dim(M)\).  
		
		
		(2) \(\Rightarrow\) (1): It suffices to show that \(\Ext_R^i(k, M) = 0\) for all \(i < \dim(M)\) when (2) holds for \(A\) satisfying \ref{ConditionGeneralPerfd} in \Cref{GeneralPerfectoidCondition}. Set $D \defeq 2(\dim(R) + 1)$.
		By \Cref{CorVanishing}, for any \(s < \dim(M) - D\), we have
		\(
		\Ext_R^{s+D}(k, M) = 0.
		\)
		As \(s\) ranges over all integers less than \(\dim(M) - D\), the exponent \(s+D\) ranges over all integers less than \(\dim(M)\). Hence \(\Ext_R^i(k, M) = 0\) for all \(i < \dim(M)\). Therefore \(\depth(M) \ge \dim(M)\). Since the reverse inequality \(\depth(M) \le \dim(M)\) always holds, we conclude that \(\depth(M) = \dim(M)\). Thus \(M\) is Cohen--Macaulay.
	\end{proof}
	
	In order to extend the preceding result to a more general setting, we must first recall the definition of Cohen--Macaulay modules without the finite-generation assumption.
	
	\begin{definition}[{cf.~\cite[Definition 2.1]{asgharzadeh2009Notion}}] \label{DefCMModule}
		Let $M$ be a (not-necessarily finitely generated) $R$-module. By $\Ext$-grade of $M$ we mean $\egrade_R(M) \defeq \egrade_R(\mfrakm, M) \defeq \inf\{i:\Ext^i_R(k,M)\neq 0\}.$
		The  dimension \(\dim(M)\) of $M$ is the supremum of lengths of chains of prime ideals in the support of $M$.  Finally, we say \(M\) is \emph{Cohen--Macaulay} precisely when its Ext-grade equals \(\dim(M)\).
	\end{definition}
	
	While part of the following  lemma has been proved in \cite[Proposition 8.2]{simon1990Homological}, we record a proof whose method is slightly different.
	
	\begin{lemma} \label{CMcomparison}
		Let \((R, \mfrakm)\) be a Noetherian local ring and let \(M\) be a nonzero complete \(R\)-module.
		Then \(M\) is (balanced) big Cohen--Macaulay \(R\)-module if and only if \(M\) is maximal Cohen--Macaulay, namely, \(M\) is Cohen--Macaulay in the sense of \Cref{DefCMModule} and \(\dim(M) = \dim(R)\).
	\end{lemma}
	
	\begin{proof} Since $M$ is separated, which means  $\bigcap_n \mfrakm^n M=0$, we have \(\mfrakm M \neq M\), because $M$ is nonzero.
		Fix any system of parameter \(\underline{x} \defeq x_1, \dots, x_d\) of \(R\).
		Then we have an equality
		\begin{equation*}
			\egrade_R(M) = \inf\set{i}{H^i(\Hom_R(\Kos(\underline{x}; R), M)) \neq 0} = \inf\set{i}{H^{d-i}(\Kos(\underline{x}; M)) \neq 0},
		\end{equation*}
		by \cite[Proposition ~(iii)]{{asgharzadeh2009Notion}} and the duality of Koszul complexes.
		The above equality shows that the equality \(\egrade_R(M) = d\) holds if and only if the sequence \(x_1^{n_1}, \dots, x_d^{n_d}\) is a Koszul-regular sequence on \(M\) for any \(n_i \geq 1\) since they are again a system of parameters.
		We will show that if \(M\) is complete, then this Koszul-regular sequence is a regular sequence: As in the proof of \citeSta{061S}, the completeness implies that \(x_1^{n_1}\) is a non-zero-divisor on \(M\) and \(x_2^{n_2}, \dots, x_d^{n_d}\) is a Koszul-regular sequence on \(M/x_1^{n_1}M\) for all \(n_i \geq 1\).
		So by induction, it suffices to show that \(M/x_1^{n_1}M\) is complete.
		We can use the derived completeness here: Since \(M\) is complete, the quotient \(M/x_1^{n_1}M\) is derived \(\mfrakm\)-complete (\citeSta{091U}).
		Since the Koszul complex \(\Kos(M/x_1^{n_1}M; x_2^{n_2}, \dots, x_d^{n_d})\) is quasi-isomorphic to the usual quotient \(M/(x_1^{n_1}, \dots, x_d^{n_d})\), we have an isomorphism
		\begin{equation*}
			M/x_1^{n_1}M \cong R\lim_{k \geq 1} \Kos(M/x_1^{n_1}M; x_2^{k}, \dots, x_d^{k}) \cong R\lim_{k \geq 1} M/(x_1^{n_1}, x_2^k, \dots, x_d^k) \cong \widehat{M/x_1^{n_1}M}
		\end{equation*}
		in \(\mcalD(R)\), where the first isomorphism follows from the derived completeness, the last isomorphism follows from the Mittag--Lefller condition, and \(\widehat{M/x_1^{n_1}M}\) is the \(\mfrakm\)-adic completion.
	\end{proof}

	\begin{corollary} \label{TheoEquivfCM}
		Let \((R, \mfrakm, k)\) be a Noetherian local ring with \(p \in \mfrakm\), and let \(M\) be an \(R\)-module so that $\mfrakm M\neq M$.
		Then the following are equivalent:
		\begin{enumerate}
			\item \(M\) is Cohen--Macaulay.
			\item There exists an \(R\)-algebra \(A\) satisfying one of the conditions in \Cref{GeneralPerfectoidCondition} such that \(\Ext_R^i(A, M) = 0\) for all \(i < \dim(M)\).
		\end{enumerate}
	\end{corollary}
	
	\begin{proof}
		This is similar to the finitely generated case (see \Cref{TheoEquivCM}),  with the understanding that the Ext-grade of \(M\) is less than or equal to \(\dim(M)\), see  \cite[Corollary 2.5]{asgharzadeh2009Notion}. 
	\end{proof}
	
	\begin{example} Adopt the notation of the preceding \Cref{TheoEquivfCM}.  The condition $\mfrakm M\neq M$ is important. Indeed, let $R\defeq  \mathbb{F}_p[[x]]$, and set $M\defeq \mathcal{Q}(R)$ be the fraction field of $R$. Clearly, $\Ext_R^i(k,M)=0$ for any $i\geq 0$. Since $M$ is 1-dimensional, it implies that  \(M\) is not Cohen--Macaulay. But, there is an \(R\)-algebra \(A\defeq k\) from \Cref{GeneralPerfectoidCondition} so that \(\Ext_R^i(A, M) = 0\) for all \(i < \dim(M)\).
	\end{example}


	\begin{corollary} \label{TheoEquivBCM}
		Let \((R, \mfrakm, k)\) be a Noetherian local ring with \(p \in \mfrakm\), and let \(M\) be a nonzero complete \(R\)-module.
		Then  \(M\) is a balanced big Cohen--Macaulay $R$-module if and only if there is an \(R\)-algebra \(A\) satisfying one of the conditions in \Cref{GeneralPerfectoidCondition} such that \(\Ext_R^i(A, M) = 0\) for all \(i < \dim(R)\).
	\end{corollary}
	
	\begin{proof}
		Recall from $\bigcap_n \mfrakm^n M=0$ that  $\mfrakm M\neq M$. Using \Cref{CMcomparison} and \Cref{TheoEquivfCM}, the claim is now immediate.
	\end{proof}

	\begin{corollary} \label{CorEquivcan}
		Let \((R, \mfrakm, k)\) be a Cohen--Macaulay complete local ring with \(p \in \mfrakm\), and let \(M\) be a nonzero finitely generated torsion-free \(R\)-module.
		Then the following are equivalent:
		\begin{enumerate}
			\item \(M \cong \omega_R^{\oplus n}\) for some \(n \geq 1\).
			\item There exists an \(R\)-algebra \(A\) satisfying one of the conditions in \Cref{GeneralPerfectoidCondition} such that
			\[
			\Ext_R^i(A, M) = 0 \quad \text{for all } i \neq \dim(M).
			\]
		\end{enumerate}
	\end{corollary}
	
	\begin{proof}
		Since \(R\) is complete, the canonical module exists. Assume (2). In particular, \(\Ext_R^i(A, M) = 0\) for all \(i < \dim(M)\), so \(M\) is maximal Cohen--Macaulay by \Cref{TheoEquivCM}, because $\dim(M)=\dim(R)$. Moreover, the vanishing \(\Ext_R^i(A, M) = 0\) for all sufficiently large \(i\) forces \(M\) to have finite injective dimension by \Cref{TheoEquivGor}. Therefore, \cite[Exercise~3.3.28~(a)]{bh} yields \(M \cong \omega_R^{\oplus n}\) for some \(n \geq 1\), proving (2)~\(\Rightarrow\)~(1).
		
		Conversely, set \(d  \defeq  \dim(\omega_R) = \depth(R)\). Since \(\idim_R(\omega_R) = d\), we have \(\Ext_R^i(A, \omega_R^{\oplus n}) \cong \Ext_R^i(A, \omega_R)^{\oplus n} = 0\) for all \(i > d\). The vanishing for \(i < d\) is precisely the content of \Cref{TheoEquivCM}. Hence (1)~\(\Rightarrow\)~(2) follows.
	\end{proof}

	Corollary \ref{CorEquivcan}~(2) gives no data on  $\Ext_R^{\dim(M)}(A, M)$.
	Here, we show its vanishing may depend on the Krull dimension of the ring.
	
	\begin{example}Adopt the notation of the preceding \Cref{CorEquivcan}, suppose \(R\) is of characteristic \(p\) and \(d \defeq \dim(R)\). The following holds.
		\begin{enumerate}
			\item One has \(d= \dim(\omega_R)\), and \(\Ext_R^d(R_{\perf}, \omega_R)\) vanishes if and only if \(d\) is positive.
			\item The torsion-free assumption on $M$ is needed. 
		\end{enumerate}
	\end{example}

	\begin{proof} (1) First, we assume $d>0$. Since \(R\) is Cohen--Macaulay, we have \(H^0_{\mfrakm}(R_{\perf}) = 0\) and the natural isomorphism
		\(H^0_{\mfrakm}(R_{\perf}) \cong \Tor_d^R(R_{\perf}, \omega_R^\vee)
		\)
		forces \(\Tor_d^R(R_{\perf}, \omega_R^\vee) = 0\). Dualizing gives $\Ext_R^d(R_{\perf}, \omega_R) \cong \Ext_R^d(R_{\perf}, \widehat{\omega_R}) \cong \Tor_d^R(R_{\perf}, \omega_R^\vee)^\vee = 0,$ see \cite[\S 4.1]{simon1990Homological}. Conversely, assume $d=0$. This may read as $\omega_R=E_R(k),$ and $R_{\perf}=k_{\perf}\cong \bigoplus_{\Lambda} k$. Consequently,
		\begin{equation*}
			\Ext_R^d(R_{\perf}, \omega_R)=\Hom_R(R_{\perf}, \omega_R)= \prod_{\Lambda}k^\vee\neq 0,
		\end{equation*}
		as desired.

		(2) Let $R\defeq  \mathbb{F}_p[[x]]$, and set $M\defeq R/xR$. 
		Since $R$ is regular, any module has finite injective dimension, so  \(
		\Ext_R^i(R_{\perf}, M) = 0 \quad \text{for all } i \gg 0.
		\)   Then by
		Corollary \ref{ExtAllVanishing} we see
		\(
		\Ext_R^i(R_{\perf}, M) = 0 \quad \text{for all } i \neq \dim(M)=0.
		\) Clearly, $M\ncong \omega_R^{\oplus n}$   for any \(n \geq 1\).
	\end{proof}

	\subsection{Regularity and injective dimensions}
	
	
	Finally, we can also give a characterization of regular local rings in terms of the \emph{injective dimension} of perfectoid algebras, whereas the previous results \cite[Theorem 4.7]{bhatt2019Regular} treat the \emph{projective dimension} of perfectoid algebras.
	In particular, the equivalence (1) \(\Leftrightarrow\) (4) in \Cref{FinInjDimPerfRegular} is a dual version of the previous result.
	
	\begin{theorem} \label{FinInjDimPerfRegular}
		Let \((R, \mfrakm, k)\) be a Noetherian local ring with \(p \in \mfrakm\). Then the following are equivalent:
		\begin{enumerate}
			\item \(R\) is regular.
			\item The perfection \((R/pR)_{\perf}\) has finite injective dimension as an \(R\)-module.
			\item There exists an \(R\)-algebra \(A\) satisfying one of the conditions in \Cref{GeneralPerfectoidCondition} such that \(A\) has finite injective dimension as an \(R\)-module.
			\item \(R\) has isolated singularity and there exists an \(R\)-algebra \(A\) satisfying one of the conditions in \Cref{GeneralPerfectoidCondition} such that \(\Ext^i_R(A, A) = 0\) for \(i \gg 0\).
		\end{enumerate}
	\end{theorem}
	
	\begin{proof}
		(1) \(\Rightarrow\) (2): If \(R\) is regular, then every \(R\)-module has finite injective dimension, and hence so does \((R/pR)_{\perf}\).
		
		(2) \(\Rightarrow\) (3): This implication is trivial, since \((R/pR)_{\perf}\) is perfectoid.
		
		(3) \(\Rightarrow\) (1): By \citeSta{0AVJ}, it suffices to show that \(R\Hom_R(k, k)\) has bounded cohomology.
		Since \((A/\mfrakm A)_{\perf}\) is a nonzero direct sum of copies of \(k\), it suffices to show the complex \(R\Hom_R(k, (A/\mfrakm A)_{\perf})\) has bounded cohomology.
		Since \(A/(f_0^{1/p^\infty})\) is assumed to be perfect, so is its reduced quotient \(A/\sqrt{\mfrakm A}\). Therefore, we have an isomorphism \((A/\mfrakm A)_{\perf} \cong A/\sqrt{\mfrakm A}\) of \(A\)-algebras, and in particular,
		\[
		R\Hom_R(k, (A/\mfrakm A)_{\perf}) \cong R\Hom_R(k, A/\sqrt{\mfrakm A}).
		\]
		Since \(A\) is assumed to have finite injective dimension as an \(R\)-module and we may assume \ref{ConditionGeneralPerfd} in \Cref{GeneralPerfectoidCondition} for \(A\), the boundedness of this object follows from \Cref{TheoFinIdim}.
		
		The implication (1) \(\Rightarrow\) (4) is trivial, so it remains to prove (4) \(\Rightarrow\) (3).
		Condition (4) implies \(\Ext^i_R(k, A) = 0\) for \(i \gg 0\) by \Cref{CorVanishing}.
		Since \(R\) is assumed to have isolated singularity, this vanishing implies that \(A\) has finite injective dimension as an \(R\)-module by \cite[Theorem VI.9]{schoutens2003Projectivea}.
	\end{proof}
	
	If such an \(R\)-algebra \(A\) exists and is a big Cohen--Macaulay \(R\)-algebra, then the injective dimension of \(A\) is strictly determined by the dimension of \(R\):
	\begin{corollary} \label{CorInjDim}
		Let \((R, \mfrakm, k)\) be a Noetherian local ring with \(p \in \mfrakm\), and let \(A\) be either \((R/pR)_{\perf}\) or an \(R\)-algebra satisfying one of the conditions in \Cref{GeneralPerfectoidCondition}.
		If \(A\) is big Cohen--Macaulay over \(R\) and \(\idim_R(A) < \infty\), then
		\(
		\idim_R(A) = \dim(R).
		\)
	\end{corollary}
	
	\begin{proof}
		By \Cref{FinInjDimPerfRegular}, \(R\) is regular, and therefore \(\idim_R(A) \leq \dim(R)\).
		On the other hand, since \(A\) is a big Cohen--Macaulay \(R\)-algebra, we have \(\Ext_R^i(k, A) = 0\) for \(i < \dim(R)\), but not for \(i = \dim(R)\).
		By \citeSta{0A5T}, this implies that \(\dim(R) \le \idim_R(A)\).
		Combining the two inequalities yields \(\idim_R(A) = \dim(R)\).
	\end{proof}
	For example, we can compute the injective dimension of the perfect closure and the absolute integral closure of a Noetherian local domain of characteristic \(p\).
	
	\begin{example}\label{epa}
		Let \(R\) be a Noetherian local ring with \(p \in \mfrakm\). From \Cref{FinInjDimPerfRegular} and \Cref{CorInjDim}, we have the following assertions:
		\begin{enumerate}
			\item If \(R\) is of characteristic \(p\) and \(\idim_R(R_{\perf})\) is finite, then \(R\) is regular and \(\idim_R(R_{\perf}) = \dim(R)\). Indeed, the finiteness implies that \(R\) is regular, and then \(R \to R_{\perf}\) is faithfully flat over \(R\) and is a big Cohen--Macaulay \(R\)-algebra.
			\item Assume \(R\) is an excellent integral domain and let \(R^+\) denote its absolute integral closure. Similarly, if \(\idim_R(\widehat{R^+})\) is finite for the \(p\)-adic completion \(\widehat{R^+}\) of \(R^+\), then \(R\) is regular and \(\idim_R(\widehat{R^+}) = \dim(R)\). Indeed, the \(p\)-adic completion of the absolute integral closure \(R^+\) is a big Cohen--Macaulay \(R\)-algebra as explained in the proof of \Cref{TorAllVanishing}.
		\end{enumerate}
	\end{example}
	
	\begin{corollary}\label{embcm}
		Let \(R\) be a complete local domain, and let \(A\) be a perfectoid \(R\)-algebra of finite injective dimension. Suppose \(A\) maps to a big Cohen--Macaulay algebra. Then
		\(
		\idim_R(A) = \dim(R).
		\)
	\end{corollary}
	
	\begin{proof}
		Let \(d = \dim R\) and let \(\mfrakm\) denote the maximal ideal of \(R\).
		By \cite[Corollary~10.6]{hochster1994Solid}, \(A\) is solid. Therefore, by \cite[Corollary~2.4]{hochster1994Solid},
		\( 
		H^d_{\mfrakm}(A) \neq 0.
		\)
		By definition of local cohomology, this gives
		\( 
		\varinjlim_n \Ext^d_R(R/\mfrakm^n, A) \neq 0.
		\)
		In particular, \(\idim_R(A) \ge d\). The reverse inequality follows by arguing as in \Cref{CorInjDim}.
	\end{proof}

	
	
\end{document}